\documentclass[com]{article}
\usepackage{relsize,amsfonts}
\usepackage{tikz}
\usepackage{framed}
\usepackage{amsmath, amssymb, amsthm}
\usepackage{graphicx}
\usepackage[all]{xy}
\newtheorem{definition}{Definition}[section]
  \newtheorem{theorem}[definition]{Theorem}
  \newtheorem{lemma}[definition]{Lemma}
  \newtheorem{corollary}[definition]{Corollary}

\newcommand{\lar}[1]{\mathlarger{\mathlarger{\mathlarger{#1}}}}

\title{Implicative models of intuitionistic set theory}
\author{Samuele Maschio}
\date{}
\begin{document}

\maketitle
\begin{abstract} In this paper we will show that using implicative algebras one can produce models of intuitionistic set theory generalizing both realizability and Heyting-valued models. This has as consequence that if one assumes the inaccessible cardinal axiom, then every topos which is obtained from a $\mathbf{Set}$-based tripos as the result of the tripos-to-topos construction hosts a model of intuitionistic set theory.
\end{abstract}
\section{Introduction} Implicative algebras were introduced by A.Miquel in \cite{miq1} in order to provide a common foundation for realizability and forcing. 

Given a complete Heyting algebra, one can define a tripos (see e.g.\ \cite{HJP}) of Heyting-valued predicates over $\mathbf{Set}$. Given a PCA one can obtain a realizability tripos, as shown e.g.\ in \cite{VOO08}. If one applies a construction called the tripos-to-topos construction (\cite{HJP}) to these triposes one obtains forcing toposes and realizability toposes, respectively.  In \cite{miq1} it is shown that both triposes are particular cases of a more general notion of tripos induced by an implicative algebra (which we call here an implicative tripos). In \cite{miq2} Miquel proved much more, namely that every set-based tripos is in fact (isomorphic to) an implicative tripos. Forcing toposes and realizability toposes host models of intuitionistic set theory $\mathbf{IZF}$ (and the same holds for a larger class of toposes as shows in \cite{MSt}), provided there exists an enough big strongly inaccessible cardinal. In the first case, the models of $\mathbf{IZF}$ are the so-called Heyting-valued models of set theory (see \cite{bell}) while in the second case they are Friedman/Rosolini/McCarty realizability models of $\mathbf{IZF}$ (see \cite{Friedman, Rosolini, McCarty}). If one takes a look to how these models are defined then it can notice some similarities. In this paper we show that these similarities comes from the fact that we can generalize the construction of Heyting-valued and realizability models of $\mathbf{IZF}$ to toposes coming from implicative triposes. We work in $\mathbf{ZFC}$ as metatheory.

\section{Intuitionistic set theory}
In the language of intuitionistic set theory $\mathbf{IZF}$ the only terms are variables and there are two binary predicate symbols: equality $=$ and membership $\in$.
As usual in the language of set theory $\forall x\in y\,\varphi$ is a shorhand for $\forall x(x\in y\rightarrow \varphi)$ and $\exists x\in y\,\varphi$ is a shorhand for $\exists x(x\in y\wedge \varphi)$, while $x\subseteq y$ is a shorthand for $\forall z\in x\,(z\in y)$.

We consider the following presentation of the axioms of $\mathbf{IZF}$.
\begin{enumerate}
\item[$\mathbf{Emp})$] $\exists x\forall y\in x\,\bot$
\item[$\mathbf{Ext})$] $\forall x\forall y\,(x\subseteq y\wedge y\subseteq x\rightarrow x=y)$ 
\item[$\mathbf{Pair})$] $\forall x\forall y\exists z\,(x\in z\wedge y\in z)$
\item[$\mathbf{Union})$] $\forall x\exists u\forall y\in x\forall z\in y\,(z\in u)$
\item[$\mathbf{Pow})$] $\forall x\exists z\forall y\,(y\subseteq x\rightarrow y\in z)$
\item[$\mathbf{Inf})$] $\exists u\mathbf{Inf}(u)$ where $\mathbf{Inf}(u)$ is the conjunction of  $\mathbf{Inf}_1(u):\equiv\exists x\in u\forall y\in x\,\bot$ and  $\mathbf{Inf}_2(u):\equiv\forall x\in u\exists y\in u(x\subseteq y\wedge x\in y\wedge \forall z\in y(z\in x \vee z=x))$
\item[$\mathbf{Sep}_{\varphi})$] $\forall w_1....\forall w_n\forall x\exists y\forall z\left(\forall z\in y\,(z\in x\wedge \varphi)\wedge \forall z\in x\,(\varphi\rightarrow z\in y)\right)$ for all formulas in context $\varphi[w_1,...w_n,x,z]$.
\item[$\mathbf{Col})$] $\forall w_1....\forall w_n\left(\forall x\in y \exists z\, \varphi \rightarrow \exists u\forall x\in y\exists z\in u\,\varphi\right)$ for all formulas in context $\varphi[w_1,...w_n,x,y,z]$.
\item[$\mathbf{\in\textrm{-}Ind}_{\varphi})$] $\forall w_1...\forall w_n(\forall x (\forall y\in x\,\varphi[y/x]\rightarrow \varphi)\rightarrow \forall x\,\varphi)$ for all formulas in context $\varphi[w_1,...,w_n,x]$ formula in context.
\end{enumerate}

\section{Implicative algebras and implicative triposes}
An implicative algebra is a $4$-tuple $\mathbb{A}=(A,\leq,\rightarrow,\Sigma)$ where 
\begin{enumerate}
\item $(A,\leq)$ is a complete lattice;
\item $\rightarrow:A\times A\rightarrow A$ is a function which is monotone in the second component and anti-monotone in the first component, and which satisfies the following condition 
$$a\rightarrow \bigwedge_{i\in I}b_i=\bigwedge_{i\in I}\left(a\rightarrow b_i
\right)$$
for every indexed family $(b_i)_{i\in I}$ of elements of $A$ and every $a\in A$;
\item $\Sigma\subseteq A$  is upward closed, it contains also $b$ as soon as it contains $a\rightarrow b$ and $a$, and it contains $\mathbf{K}:=\bigwedge_{a,b\in A}(a\rightarrow (b\rightarrow a))$ and $\mathbf{S}:=\bigwedge_{a,b\in A}((a\rightarrow (b\rightarrow c))\rightarrow ((a\rightarrow b)\rightarrow (a\rightarrow c)))$.
\end{enumerate}

\emph{Every complete Heyting algebra $(H,\leq)$ with Heyting implication $\rightarrow$ gives rise to an implicative algebra $(H,\leq,\rightarrow,\{\top\})$. Moreover, every total combinatory algebra $(R,\cdot)$ gives rise to an implicative algebra $(\mathcal{P}(R),\subseteq,\Rightarrow, \mathcal{P}(R)\setminus \{\emptyset\})$, where $A\Rightarrow B:=\{r\in R|\; r\cdot a \in B \textrm{ for every }a\in A\}$ for every $A,B\subseteq R$. Many other examples can be found in \cite{miq1}. In the case of a partial combinatory algebra $(R,\cdot)$, the $4$-uple $(\mathcal{P}(R),\subseteq,\Rightarrow, \mathcal{P}(R)\setminus \{\emptyset\})$ is not in general an implicative algebra, but a quasi-implicative algebra (see \cite{miq1}). However, there is a standard way to transform it into an implicative algebra in such a way that the tripos one obtains is equivalent to the realizability tripos built from $(R,\cdot)$.}\\

Closed $\lambda$-terms with constant parameters in $\mathbb{A}$ can be encoded in an implicative algebra as follows: $a^{\mathbb{A}}:=a$ for every $a\in A$, $(ts)^{\mathbb{A}}:=t^{\mathbb{A}}\cdot s^{\mathbb{A}}$ and $(\lambda x.t)^{\mathbb{A}}:=\bigwedge_{a\in A}\left(a\rightarrow (t[a/x])^{\mathbb{A}}\right)$ where application $\cdot$ is defined as follows for every $a,b\in A$:
$$a\cdot b:=\bigwedge\{x\in A|\,a\leq b\rightarrow x\}$$
Under this encoding, if we define $\mathbf{k}$ as $\lambda x.\lambda y.x$ and $\mathbf{s}$ as $\lambda x.\lambda y.\lambda z.xz(yz)$, one can show that $\mathbf{K}=\mathbf{k}^{\mathbb{A}}$ and $\mathbf{S}=\mathbf{s}^{\mathbb{A}}$.

Useful properties of the encoding of $\lambda$-terms in $\mathbb{A}$ are the following:
\begin{enumerate}
\item if $t$ $\beta$-reduces to $s$, then $t^{\mathbb{A}}\leq s^{\mathbb{A}}$;
\item if $t$ is a pure-$\lambda$ term with free variables $x_1,...,x_n$ and $a_1,...,a_n\in \Sigma$, then $(t[x_1:a_1,...,x_n:a_n])^{\mathbb{A}}\in \Sigma$\footnote{We denote with $t[x_1:a_1,...,x_n:a_n]$ the $\lambda$-term obtained by substituting the variables $x_1,...,x_n$ with $a_1,...,a_n$.}; in particular the encodings of closed pure $\lambda$-terms are elements of $\Sigma$.
\end{enumerate}
In what follows we will remove the superscript $\mathbb{A}$ from the encoding of $\lambda$-terms in order to lighten the notation.

Moreover for every $a,b\in A$ following \cite{miq1} we define:
$$a\times b:=\bigwedge_{x\in A}\left((a\rightarrow (b\rightarrow x))\rightarrow x\right)$$
$$a+ b:=\bigwedge_{x\in A}\left((a\rightarrow x)\rightarrow ((b\rightarrow x)\rightarrow x)\right)$$
and for every set indexed family $(a_{i})_{i\in I}$ we define
$$\lar{\forall} _{i\in I} a_i:=\bigwedge_{i\in I}a_{i}\qquad\lar{\exists} _{i\in I} a_i:=\bigwedge_{x\in A}\left(\bigwedge_{i\in I}\left(a_{i}\rightarrow x\right)\rightarrow x\right)$$

We introduce the following shorthands for some $\lambda$-terms: $\mathbf{k}:=\lambda x.\lambda y.x$, $\overline{\mathbf{k}}:=\lambda x.\lambda y.y$, $\mathbf{p}:=\lambda x.\lambda y.\lambda z.zxy$, $\mathbf{p}_1:=\lambda x.x\mathbf{k}$, $\mathbf{p}_2:=\lambda x.x\overline{\mathbf{k}}$, $\mathbf{j}_1:=\lambda x.\lambda z.\lambda w.zx$, $\mathbf{j}_2:=\lambda x.\lambda z.\lambda w.wx$, $\mathbf{e}:=\lambda x.\lambda z.zx$. Notice that (the encoding of) all of them belong to $\Sigma$.

If $\Gamma$ is a finite list of variable assignments $x_1:a_1,...,x_n:a_n$ with $a_1,...,a_n\in A$ and $t$ is a $\lambda$-term with parameters in $A$ and free variables among $x_1,..,x_n$, we write $\Gamma\vdash t:a$ as a shorthand for $t[\Gamma]^{\mathbb{A}}\leq a$ (where $t[\Gamma]$ is the result of the substitution corresponding to $\Gamma$ applied to $t$) and the following rules are sound (this is a little variation of the system of rules presented in \cite{miq1}).
\begin{framed}
{\small
$$\cfrac{x:a\in \Gamma}{\Gamma\vdash x:A}\;\qquad \cfrac{}{\Gamma\vdash a:a}\qquad \cfrac{\Gamma\vdash t:a\qquad  a\leq b}{\Gamma\vdash t:b}$$
$$\cfrac{\Gamma\vdash t:\bot}{\Gamma\vdash t:a}\qquad \cfrac{\Gamma\vdash t:a}{\Gamma\vdash t:\top}\qquad\cfrac{\Gamma\vdash t:a\rightarrow b\qquad \Gamma\vdash s:a}{\Gamma\vdash ts:b}\qquad \cfrac{\Gamma,x:a\vdash t:b }{\Gamma\vdash \lambda x.t:a\rightarrow b}$$
$$\cfrac{\Gamma\vdash t:a\qquad \Gamma\vdash s:b}{\Gamma\vdash \mathbf{p}ts: a\times b}\qquad \cfrac{\Gamma\vdash t:a\times b}{\Gamma\vdash\mathbf{p}_1t:a}\qquad  \cfrac{\Gamma\vdash t:a\times b}{\Gamma\vdash\mathbf{p}_2t:b}$$
$$\cfrac{\Gamma\vdash t:a}{\Gamma\vdash\mathbf{j}_1t:a+b}\;\;\cfrac{\Gamma\vdash t:b}{\Gamma\vdash\mathbf{j}_2t:a+b}\;\; \cfrac{\Gamma\vdash t:a+b\qquad \Gamma, x:a\vdash u:c\qquad\Gamma, y:b\vdash v:c }{\Gamma\vdash t(\lambda x.u)(\lambda y.v):c}$$
$$\cfrac{\Gamma\vdash t:a_i\,(\textrm{for all }i\in I)}{\Gamma\vdash t: \lar{\forall} _{i\in I}a_i}\qquad \cfrac{\Gamma\vdash t: \lar{\forall} _{i\in I}a_i}{\Gamma\vdash t:a_{\overline{i}}}\;\overline{i}\in I$$
$$\cfrac{\Gamma\vdash t:a_{\overline{i}}}{\Gamma\vdash 	\mathbf{e}t:\lar{\exists} _{i\in I}a_i }\qquad \cfrac{\Gamma\vdash t:\lar{\exists} _{i\in I}a_i\qquad \Gamma,x:a_i\vdash u:b\,(\textrm{ for all }i\in I)}{\Gamma\vdash t(\lambda x.u):b}$$
}
\end{framed}
As shown in \cite{miq1}, to every implicative algebra $\mathbb{A}$ can be associated a tripos (see \cite{HJP} or \cite{VOO08}) 
$$\mathbf{p}_{\mathbb{A}}:\mathbf{Set}^{op}\rightarrow \mathbf{Heyt}$$
by sending every set $I$ to the posetal reflection of the preordered set $(A^I,\vdash_{\Sigma[I]})$ where $\varphi\vdash_{\Sigma[I]}\psi$ if and only if $\bigwedge_{i\in I}(\varphi(i)\rightarrow \psi(i))\in \Sigma$ and every function $f:I\rightarrow J$ to the function induced by the pre-composition function $(-)\circ f:A^J\rightarrow A^I$.
Component-wise use of $\rightarrow$, $\times$ and $+$ defines a Heyting prealgebra structure on every preorder $(A^I,\vdash_{\Sigma[I]})$, which is preserved by precomposition.
$\lar{\exists}$ and $\lar{\forall}$ are used to produce left and right adjoints to reindexing maps satisfying Beck-Chevalley condition, while a generic predicate is given by (the equivalence class of) the identity function on $A$. 

A remarkable result in \cite{miq2} is the following
\begin{theorem}\label{teomiq}
Let $\mathbf{p}:\mathbf{Set}^{op}\rightarrow \mathbf{Heyt}$ be a tripos. Then, there exists an implicative algebra $\mathbb{A}$ such that $\mathbf{p}$ is isomorphic to $\mathbf{p}_{\mathbb{A}}$.
\end{theorem}

Recall also (see e.g.\ \cite{VOO08}) that to every tripos $\mathbf{p}$ over $\mathbf{Set}$ is associated an elementary topos $\mathbf{Set}[\mathbf{p}]$ obtained by means of the so-called ``tripos-to-topos'' construction.

\section{Implicative-valued models of $\mathbf{IZF}$-$\mathbf{Col}$}
{\bf Let us assume a strongly inaccessible cardinal $\kappa$ exists and let $\mathbb{A}$ be a fixed implicative algebra such that $|A|<\kappa$.}

We define the following hierarchy of sets indexed by ordinals:
$$W_\alpha^{\mathbb{A}}:=\begin{cases} 
\emptyset\textrm{ if }\alpha=0\\
\mathsf{Part}(W_{\beta}^{\mathbb{A}},A)\textrm{ if }\alpha=\beta+1\\
\bigcup_{\beta<\alpha}W_{\beta}^{\mathbb{A}}\textrm{ if }\alpha\textrm{ is a limit ordinal}\\
\end{cases}$$
where $\mathsf{Part}(X,Y)$ denotes the set of partial functions from $X$ to $Y$.
We take $\mathbf{W}$ to be $W^{\mathbb{A}}_{\kappa}$.
Since $W_{\alpha}^{\mathbb{A}}\subseteq W_{\beta}^{\mathbb{A}}$ if $\alpha<\beta$, one can assign a rank in the hierarchy to every element of $\mathbf{W}$ in the obvious way.
In particular, we can define simultaneously, by recursion on rank, two functions $\in_\mathbf{W}, =_{\mathbf{W}}:\mathbf{W}\times \mathbf{W}\rightarrow A$:
\begin{enumerate}
\item $\alpha\in_{\mathbf{W}}\beta:=\lar{\exists} _{t\in \partial_0(\beta)}\left( \beta(t)\times t=_{\mathbf{W}}\alpha \right)$
\item $\alpha=_{\mathbf{W}}\beta:=\alpha\subseteq_{\mathbf{W}}\beta\times\beta\subseteq_{\mathbf{W}}\alpha $
where $\alpha\subseteq_{\mathbf{W}}\beta:=\lar{\forall} _{t\in \partial_0(\alpha)}\left(\alpha(t)\rightarrow t\in_\mathbf{W}\beta\right)$.
\end{enumerate}
We interpret the language of set theory in such a way that to every formula in context $\varphi[x_{1},...,x_n]$ we associate a function $$\left\|\varphi[x_{1},...,x_n]\right\|:\mathbf{W}^{n}\rightarrow A$$
by recursion on complexity of formulas as follows:
\begin{enumerate}
\item $\left\|x_i\in x_j\,[x_1,...,x_n]\right\|(\alpha_1,...,\alpha_n):\equiv \alpha_i\in_{\mathbf{W}}\alpha_j$
\item $\left\|x_i= x_j\,[x_1,...,x_n]\right\|(\alpha_1,...,\alpha_n):\equiv\alpha_i=_{\mathbf{W}}\alpha_j$
\item $\left\|\varphi\wedge \psi[\underline{x}]\right\|(\underline{\alpha}):\equiv \left\|\varphi[\underline{x}]\right\|(\underline{\alpha})\times \left\|\psi[\underline{x}]\right\|(\underline{\alpha})$
\item $\left\|\varphi\vee \psi[\underline{x}]\right\|(\underline{\alpha}):\equiv \left\|\varphi[\underline{x}]\right\|(\underline{\alpha})+ \left\|\psi[\underline{x}]\right\|(\underline{\alpha})$
\item $\left\|\varphi\rightarrow \psi[\underline{x}]\right\|(\underline{\alpha}):\equiv \left\|\varphi[\underline{x}]\right\|(\underline{\alpha})\rightarrow \left\|\psi[\underline{x}]\right\|(\underline{\alpha})$
\item $\left\|\exists y\,\varphi\,[\underline{x}]\right\|(\underline{\alpha}):\equiv \lar{\exists} _{\beta\in \mathbf{W}}\left(\left\|\varphi\,[\underline{x},y]\right\|(\underline{\alpha},\beta)\right)$
\item $\left\|\forall y\,\varphi\,[\underline{x}]\right\|(\underline{\alpha}):\equiv \lar{\forall} _{\beta\in \mathbf{W}}\left(\left\|\varphi\,[\underline{x},y]\right\|(\underline{\alpha},\beta)\right)$\footnote{In these two clauses, we can assume, without loss of generality, that $y$ is not a variable appearing in the context $[\underline{x}]$.}
\end{enumerate}
We write $\mathbf{W}\vDash \varphi\,[\underline{x}]$ if $\bigwedge_{\underline{\alpha}\in \mathbf{W}^{\ell(\underline{x})}}\left(\left\|\varphi\,[\underline{x}]\right\|(\underline{\alpha})\right)\in \Sigma$ when $[\underline{x}]$ is non-empty and for every statement $\varphi$ we write $\mathbf{W}\vDash \varphi$ if $\left\|\varphi\,[\,]\right\|\in \Sigma$, that is $\mathbf{W}\vDash \varphi\,[\underline{x}]$ means $\left\|\varphi[\underline{x}]\right\|$ is in the maximum class of $\mathbf{p}_{\mathbb{A}}(\mathbf{W}^{n})$ where $n$ is the lenght of $[\underline{x}]$. 

We will often write $\left\|\varphi\right\|$ instead of $\left\|\varphi[\,]\right\|$.\\

Notice that the definition of the interpretation we proposed above coincides with that of Heyting-valued models in the case in which $\mathbb{A}=(A,\leq,\rightarrow, \{\top\})$ where $(A,\leq)$ is a complete Heyting algebra with Heyting implication $\rightarrow$, while it coincides with the realizability interpretation mentioned in the introduction when $\mathbb{A}=(\mathcal{P}(R),\subseteq,\Rightarrow, \mathcal{P}(R)\setminus\{\emptyset\})$ with $(R,\cdot)$ a total combinatory algebra and $\Rightarrow$ is the usual implication between sets of realizers. The case of a partial combinatory algebra can be recovered by using the techniques illustrated in \cite{miq1}.

\subsection{Useful lemmas}

\begin{lemma}\label{not}
There exist $\mathbf{\rho}, \mathbf{j},\mathbf{\sigma},\mathbf{s}_1,\mathbf{s}_2,\mathbf{s}_3\in \Sigma$ such that 
\begin{enumerate}
\item $\mathbf{\rho}\leq \bigwedge_{\alpha\in \mathbf{W}}\left(\alpha=_{\mathbf{W}}\alpha\right)$
\item $\mathbf{j}\leq \bigwedge_{\alpha\in \mathbf{W}}\bigwedge_{u\in \partial_0(\alpha)}\left(\alpha(u)\rightarrow u\in_{\mathbf{W}}\alpha\right)$
\item $\mathbf{\sigma}\in \bigwedge_{\alpha,\beta\in \mathbf{W}}\left(\alpha=_{\mathbf{W}}\beta\rightarrow \beta=_{\mathbf{W}}\alpha\right)$
\item $\mathbf{s}_1\leq \bigwedge_{\alpha,\beta,\gamma\in \mathbf{W}}\left(\alpha=_{\mathbf{W}}\beta\times \gamma\in_{\mathbf{W}} \alpha\rightarrow \gamma\in_{\mathbf{W}}\beta\right)$
\item $\mathbf{s}_2\leq \bigwedge_{\alpha,\beta,\gamma\in \mathbf{W}}\left(\alpha=_{\mathbf{W}}\beta\times \alpha\in_{\mathbf{W}} \gamma\rightarrow \beta\in_{\mathbf{W}}\gamma\right)$
\item $\mathbf{s}_3\leq \bigwedge_{\alpha,\beta,\gamma\in \mathbf{W}}\left(\alpha=_{\mathbf{W}}\beta\times \gamma=_{\mathbf{W}} \alpha\rightarrow \gamma=_{\mathbf{W}}\beta\right)$
\end{enumerate}
\end{lemma}
\begin{proof}
\begin{enumerate}
\item Let $\mathbf{\rho}$ be $\mathbf{y}f\in \Sigma$ where $f:=\lambda r.\mathbf{p}(\lambda x.\mathbf{e}(\mathbf{p}xr))(\lambda x.\mathbf{e}(\mathbf{p}xr))$ and $\mathbf{y}$ is a pure closed $\lambda$-term which is a fixed point operator (see e.g.\ \cite{VOO08}) for which $\mathbf{y}f$ $\beta$-reduces to $f(\mathbf{y}f)$. We claim that $\mathbf{\rho}\leq \alpha=_{\mathbf{W}}\alpha$ for every $\alpha\in \mathbf{W}$. Let $\alpha$ be an arbitrary element of $\mathbf{W}$ and let us assume that $\mathbf{\rho}\leq \beta=_{\mathbf{W}}\beta$ for every $\beta\in \mathbf{W}$ with rank in the hierarchy less than that of  $\alpha$. Then we can consider the following derivation tree in which we used only rules from the previous section.
{\small
$$\cfrac{\cfrac{\cfrac{\cfrac{\cfrac{\cfrac{}{x:\alpha(u)\vdash x:\alpha(u)\;(\textrm{for all }u\in \partial_{0}(\alpha))}\qquad \cfrac{}{x:\alpha(u)\vdash \mathbf{\rho}:u=_{\mathbf{W}}u\;(\textrm{for all }u\in \partial_{0}(\alpha))}}{x:\alpha(u)\vdash \mathbf{p}x\mathbf{\rho}: \alpha(u)\times u=_{\mathbf{W}}u \;(\textrm{for all }u\in \partial_{0}(\alpha))}}{x:\alpha(u)\vdash \mathbf{e}(\mathbf{p}x\mathbf{\rho}): u\in_{\mathbf{W}}\alpha \;(\textrm{for all }u\in \partial_{0}(\alpha))}}{\vdash\lambda x.\mathbf{e}(\mathbf{p}x\mathbf{\rho}):\alpha(u)\rightarrow u\in_{\mathbf{W}}\alpha \;(\textrm{for all }u\in \partial_{0}(\alpha)) }}{\vdash\lambda x.\mathbf{e}(\mathbf{p}x\mathbf{\rho}):\alpha\subseteq_{\mathbf{W}}\alpha}}{\vdash\mathbf{p}(\lambda x.\mathbf{e}(\mathbf{p}x\mathbf{\rho}))(\lambda x.\mathbf{e}(\mathbf{p}x\mathbf{\rho})):\alpha=_{\mathbf{W}}\alpha}$$
}
The last $\lambda$-term in the deduction tree is a $\beta$ reduction of $\mathbf{\rho}$. Thus we can conclude that $\mathbf{\rho}\leq \alpha=_{\mathbf{W}}\alpha$.

\item Let $\mathbf{j}$ be defined as $\lambda x.\mathbf{e}(\mathbf{p}x\mathbf{\rho})\in \Sigma$. Assume $\alpha\in \mathbf{W}$ and $u\in \partial_{0}(\alpha)$. Then $x: \alpha(u)\vdash\mathbf{p}x\mathbf{\rho}: \alpha(u)\times u=_{\mathbf{W}}u$ . Hence, $x: \alpha(u)\vdash \mathbf{e}(\mathbf{p}x\mathbf{\rho}): u\in_{\mathbf{W}}\alpha$. Thus $$\mathbf{j}\leq \bigwedge_{\alpha\in \mathbf{W}}\bigwedge_{u\in \partial_0(\alpha)}\left(\alpha(u)\rightarrow u\in_{\mathbf{W}}\alpha\right)$$
\item $\mathbf{\sigma}$ can be just defined as $\lambda x.\mathbf{p}(\mathbf{p}_1x)(\mathbf{p}_2x) \in \Sigma$
\item[4,5,6.] Assume $\mathbf{s}_3$ to exist. 

Let $\alpha,\beta,\gamma\in \mathbf{W}$ and let $\Gamma(u)$ be a shorthand for 
$$x:\alpha=_{\mathbf{W}}\beta\times \alpha\in_{\mathbf{W}}\gamma, y:\gamma(u)\times u=_{\mathbf{W}}\alpha$$ where $u$ an arbitrary element of the domain of $\gamma$. 
Since 
$$\vdash \mathbf{s}_3:\alpha=_{\mathbf{W}}\beta\times u=_{\mathbf{W}}\alpha\rightarrow u=_{\mathbf{W}}\beta$$
we obtain $\Gamma(u)\vdash \mathbf{p}(\mathbf{p}_1y)(\mathbf{s}_3(\mathbf{p}(\mathbf{p}_1x)(\mathbf{p}_2y))):\gamma(u)\times u=_{\mathbf{W}}\beta $ from which it follows that 
$$\Gamma(u)\vdash \mathbf{e}(\mathbf{p}(\mathbf{p}_1y)(\mathbf{s}_3(\mathbf{p}(\mathbf{p}_1x)(\mathbf{p}_2y)))):\beta\in_{\mathbf{W}}\gamma$$
Since $x:\alpha=_{\mathbf{W}}\beta\times \alpha\in_{\mathbf{W}}\gamma\vdash \mathbf{p}_2 x:\gamma\in_{\mathbf{W}}\alpha$, we get
 $$x:\alpha=_{\mathbf{W}}\beta\times \alpha\in_{\mathbf{W}}\gamma\vdash (\mathbf{p}_2x)(\lambda y.\mathbf{e}(\mathbf{p}(\mathbf{p}_1y)(\mathbf{s}_3(\mathbf{p}(\mathbf{p}_1x)(\mathbf{p}_2y))))):\beta\in_{\mathbf{W}}\gamma$$
 from which it follows that 
 $$\vdash\lambda x.(\mathbf{p}_2x)(\lambda y.(\mathbf{e}(\mathbf{p}(\mathbf{p}_1y)(\mathbf{s}_3(\mathbf{p}(\mathbf{p}_1x)(\mathbf{p}_2y)))))):\alpha=_{\mathbf{W}}\beta\times \alpha\in_{\mathbf{W}}\gamma\rightarrow \beta\in_{\mathbf{W}} \gamma$$
From this it follows that $\mathbf{s}_2$ can be defined as $\lambda x.(\mathbf{p}_2x)(\lambda y.(\mathbf{e}(\mathbf{p}(\mathbf{p}_1y)(\mathbf{s}_3(\mathbf{p}(\mathbf{p}_1x)(\mathbf{p}_2y))))))$.

Assume now $\mathbf{s}_2$ to exist and consider $\Gamma'(u)$ a shorthand for 
$$x:\alpha=_{\mathbf{W}}\beta\times \gamma\in_{\mathbf{W}}\alpha,y:\alpha(u)\times u=_{\mathbf{W}}\gamma$$ where $u$ is an arbitrary element of the domain of $\alpha$. We easily see that $\Gamma'(u)\vdash \mathbf{p}(\mathbf{p}_2y)((\mathbf{p_1}(\mathbf{p}_1x))(\mathbf{p}_1y)):u=_{\mathbf{W}}\gamma \times u\in_{\mathbf{W}}\beta $. Thus
$$\Gamma'(u)\vdash\mathbf{s}_2( \mathbf{p}(\mathbf{p}_2y)((\mathbf{p_1}(\mathbf{p}_1x))(\mathbf{p}_1y))):\gamma\in_{\mathbf{W}}\beta $$
Since $x:\alpha=_{\mathbf{W}}\beta\times \gamma\in_{\mathbf{W}}\alpha\vdash \mathbf{p}_2x:\gamma\in_{\mathbf{W}}\alpha$, we have that
$$x:\alpha=_{\mathbf{W}}\beta\times \gamma\in_{\mathbf{W}}\alpha\vdash (\mathbf{p}_2x)(\lambda y.\mathbf{s}_2( \mathbf{p}(\mathbf{p}_2y)((\mathbf{p_1}(\mathbf{p}_1x))(\mathbf{p}_1y)))):\gamma\in_{\mathbf{W}}\beta$$
from which it follows that 
$$\vdash \lambda x.(\mathbf{p}_2x)(\lambda y.\mathbf{s}_2( \mathbf{p}(\mathbf{p}_2y)((\mathbf{p_1}(\mathbf{p}_1x))(\mathbf{p}_1y)))):\alpha=_{\mathbf{W}}\beta\times \gamma\in_{\mathbf{W}}\alpha\rightarrow\gamma\in_{\mathbf{W}}\beta$$
Thus $\mathbf{s}_1$ can be defined as $\lambda x.(\mathbf{p}_2x)(\lambda y.\mathbf{s}_2( \mathbf{p}(\mathbf{p}_2y)((\mathbf{p_1}(\mathbf{p}_1x))(\mathbf{p}_1y))))$.

Similarly, one can prove that if $\mathbf{s}_1$ is assumed to exist, then one can define $\mathbf{s}_3$ as a $\lambda$-term containing $\mathbf{s}_1$ as the unique parameter.

Thus it is sufficient to compose this mutual dependence to define by a fix point $\mathbf{y}g$ one among $\mathbf{s}_1$, $\mathbf{s}_2$ and $\mathbf{s}_3$ and then define the other two using that one exploiting the interdefinability. So if we define $\mathbf{s}_1$ as a fixpoint, we can then define $\mathbf{s}_3$ using $\mathbf{s}_1$ and then $\mathbf{s}_2$ using $\mathbf{s}_3$.

\end{enumerate}
\end{proof}

\begin{lemma}\label{subfor}For every formula in context $\varphi\,[\underline{x}]$ where $\underline{x}$ has length $n$, there exists $\mathbf{r}^{\varphi[\underline{x}]}\in \Sigma$ such that 
$$\mathbf{r}^{\varphi[\underline{x}]}\leq \bigwedge_{\underline{\alpha}\in \mathbf{W}^{n}}\bigwedge_{\underline{\beta}\in \mathbf{W}^{n}}\left( \underline{\alpha}=_{\mathbf{W}}\underline{\beta}\times \left\|\varphi\,[\underline{x}]\right\|(\underline{\alpha})\rightarrow \left\|\varphi\,[\underline{x}]\right\|(\underline{\beta})\right)$$
\end{lemma}
\begin{proof}
By induction on complexity of formulas by using the previous lemma for the atomic cases.
\end{proof}
Also the following lemma can easily been proved as a consequence of the previous result and of the rules in the previous section.
\begin{lemma} Let $\varphi[\underline{x}]$ and $\psi[\underline{x}]$ be formulas in context in the language of set theory and let $n$ be the lenght of $[\underline{x}]$,  
If $\varphi\vdash^{\underline{x}}_{\mathbf{IL}^=}\psi$, then $\left\|\varphi\,[\underline{x}]\right\|\vdash_{\Sigma[\mathbf{W}^{n}]}\left\|\psi\,[\underline{x}]\right\|$.
\end{lemma}

\begin{lemma}\label{rest}If $[\underline{x}]$ has lenght $n$, then
$$\left\|\exists z\in y\,\varphi\,[\underline{x},y]\right\|\equiv_{\Sigma[\mathbf{W}^{n+1}]}\Lambda \underline{\alpha}.\Lambda \beta.\lar{\exists} _{u\in \partial_{0}(\beta)}\left(\beta(u)\times \left\|\varphi\,[\underline{x},y,z]\right\|(\underline{\alpha},\beta,u) \right)\footnote{We use the notation $\Lambda \alpha.f(\alpha)$ to denote the function sending each $\alpha$ in the domain to $f(\alpha)$.}$$
$$\left\|\forall z\in y\,\varphi\,[\underline{x},y]\right\|(\underline{\alpha},\beta)\equiv_{\Sigma[\mathbf{W}^{n+1}]}\Lambda \underline{\alpha}.\Lambda \beta.\lar{\forall} _{u\in \partial_{0}(\beta)}\left(\beta(u)\rightarrow \left\|\varphi\,[\underline{x},y,z]\right\|(\underline{\alpha},\beta,u) \right)$$
\end{lemma}
\begin{proof}
We consider the case of the existential quantifier and we leave the analogous proof of the universal case to the reader. We also restrict to the case in which $\underline{x}$ is empty. The general case is analogous, but only heavier in notation.
By definition of the interpretation we have that 
$\left\|\exists z\in y\,\varphi\,[y]\right\|(\beta)$ is $$\eta:=\lar{\exists} _{\gamma \in \mathbf{W}}\left(\lar{\exists} _{u\in \partial_{0}(\beta)}(\beta(u)\times u=_{\mathbf{W}}\gamma)\times \left\|\varphi\,[y,z]\right\|(\beta,\gamma)\right)$$
We denote with $\eta_{1}(\gamma)$ the scope of the quantifier $\lar{\exists} _{\gamma \in \mathbf{W}}$, while we denote with $\eta_2(u,\gamma)$ the scope of the quantifier $\lar{\exists} _{u\in \partial_{0}(\beta)}$.
It is immediate to check that the following sequent holds for every $\beta,\gamma\in \mathbf{W}$ and every $u\in \partial_0(\beta)$: 
$$x:\eta, y:\eta_1(\gamma),z:\eta_2(u,\gamma)\vdash \mathbf{l}:(\beta=_{\mathbf{W}}\beta\times \gamma=_{\mathbf{W}}u)\times \left\|\varphi\,[y,z]\right\|(\beta,\gamma) )$$
where $\mathbf{l}:=\mathbf{p}(\mathbf{p}\rho(\sigma(\mathbf{p}_2z)))(\mathbf{p}_2y)$.
With the notation from Lemma \ref{subfor}, we can conclude that $x:\eta, y:\eta_1(\gamma),z:\eta_2(u,\gamma)\vdash \mathbf{r}^{\varphi[y,z]}\mathbf{l}:\left\|\varphi\,[y,z]\right\|(\beta,u) $. Thus $x:\eta, y:\eta_1(\gamma),z:\eta_2(u,\gamma)\vdash \mathbf{p}(\mathbf{p}_1z)(\mathbf{r}^{\varphi[y,z]}\mathbf{l}):\beta(u)\times\left\|\varphi\,[y,z]\right\|(\beta,u) $.
Hence
$$x:\eta, y:\eta_1(\gamma),z:\eta_2(u,\gamma)\vdash \mathbf{e}(\mathbf{p}(\mathbf{p}_1z)(\mathbf{r}^{\varphi[y,z]}\mathbf{l})):\exists_{w\in \partial_{0}}(\beta(w)\times\left\|\varphi\,[y,z]\right\|(\beta,w))$$
Using the rules of elimination of existential quantification, one can conclude that 
$$\eta\rightarrow \exists_{w\in \partial_{0}}(\beta(w)\times\left\|\varphi\,[y,z]\right\|(\beta,w))\geq \mathbf{l}'\in \Sigma$$
where $\mathbf{l}':=\lambda x.x\lambda y.(\mathbf{p}_1y)(\lambda z.\mathbf{e}(\mathbf{p}(\mathbf{p}_1z)(\mathbf{r}^{\varphi[y,z]}\mathbf{l})))$.

It is easier to show that 
$$ \exists_{w\in \partial_{0}}(\beta(w)\times\left\|\varphi\,[y,z]\right\|(\beta,w))\rightarrow \eta \geq \lambda x.x\lambda y.\mathbf{p}(\mathbf{p}(\mathbf{p}_1y)\rho)(\mathbf{p}_2y)\in \Sigma$$
 One can in fact write a deduction tree in which the existential quantifiers of the consequent are both witnessed by a $w\in \partial_{0}(\beta)$ for which $\beta(w)\times\left\|\varphi\,[y,z]\right\|(\beta,w)$ is assumed to hold.

\end{proof}
\begin{corollary}\label{sub} $\left\|x\subseteq y\,[x,y]\right\|\equiv_{\Sigma[\mathbf{W}^2]}\Lambda \alpha.\Lambda \beta.\left(\alpha\subseteq_{\mathbf{W}}\beta\right)$ for every $\alpha,\beta\in \mathbf{W}$.
\end{corollary}

\subsection{Validity of axioms}
\subsubsection{Empty-set}
Thanks to Lemma \ref{rest} we know that
$$\left\|\mathbf{Emp}\right\|\equiv_{\Sigma}\lar{\exists} _{\alpha\in \mathbf{W}}\lar{\forall} _{u\in \partial_{0}(\alpha)}\left(\alpha(u)\rightarrow \bot\right)$$
Consider $\emptyset\in \mathbf{W}$. Then $\lar{\forall} _{u\in \partial_{0}(\emptyset)}\left(\emptyset(u)\rightarrow \bot\right)=\bigwedge \emptyset=\top$.
This entails that $\mathbf{e}\top\leq \lar{\exists} _{\alpha\in \mathbf{W}}\lar{\forall} _{u\in \partial_{0}(\alpha)}\left(\alpha(u)\rightarrow \bot\right)$. Since $\mathbf{e}\top\in \Sigma$, we can conclude that $\mathbf{W}\vDash \mathbf{Emp}$.
\subsubsection{Extensionality}
Thanks to Corollary \ref{sub} we know that
$$\left\|\mathbf{Ext}\right\|\equiv_{\Sigma}\lar{\forall} _{\alpha\in \mathbf{W}}\lar{\forall} _{\beta\in \mathbf{W}}\left(\alpha\subseteq_{\mathbf{W}}\beta \times \beta\subseteq_{\mathbf{W}}\alpha\rightarrow \alpha=_{\mathbf{W}}\beta\right)\geq \lambda x.x \in \Sigma$$
Thus, $\mathbf{W}\vDash \mathbf{Ext}$.
\subsubsection{Pair}
$$\left\|\mathbf{Pair}\right\|\equiv_{\Sigma}\lar{\forall} _{\alpha\in \mathbf{W}}\lar{\forall} _{\beta\in \mathbf{W}}\lar{\exists} _{\gamma\in \mathbf{W}}\left(\alpha\in_{\mathbf{W}}\gamma \times \beta\in_{\mathbf{W}}\gamma\right)$$
Let us consider arbitrary $\alpha,\beta\in \mathbf{W}$ and the partial function $\eta_{\alpha,\beta}\in \mathbf{W}$ defined as follows: $\partial_{0}(\eta_{\alpha,\beta})=\{\alpha, \beta\}$ and $\eta_{\alpha,\beta}(u)=\top$ for every $u$ in the domain. By Lemma \ref{not}, $\vdash\mathbf{\rho}: \alpha=_{\mathbf{W}}\alpha$, from which it follows that 
$$\vdash\mathbf{q}:=\mathbf{e}(\mathbf{p}\top\rho): \exists_{t\in \{\alpha,\beta\}}\left(\top\times t=_{\mathbf{W}}\alpha\right)=\alpha\in_{\mathbf{W}}\eta_{\alpha,\beta}$$ In the same way, we can show that $\vdash\mathbf{q}: \beta\in_{\mathbf{W}}\eta_{\alpha,\beta}$. As a consequence 
$$\vdash\mathbf{q}':=\mathbf{p}\mathbf{q}\mathbf{q}: \alpha\in_{\mathbf{W}}\eta_{\alpha,\beta}\times\beta\in_{\mathbf{W}}\eta_{\alpha,\beta}$$ Thus,  $\vdash\mathbf{e}\mathbf{q}': \lar{\exists} _{\gamma\in \mathbf{W}}\left(\alpha\in_{\mathbf{W}}\gamma \times \beta\in_{\mathbf{W}}\gamma\right)$.  Since $\mathbf{e}\mathbf{q}'$ does not depend on $\alpha$ and $\beta$ and belongs to $\Sigma$,  
$$\vdash\mathbf{e}\mathbf{q}': \lar{\forall} _{\alpha\in \mathbf{W}}\lar{\forall} _{\beta\in \mathbf{W}}\lar{\exists} _{\gamma\in \mathbf{W}}\left(\alpha\in_{\mathbf{W}}\gamma \times \beta\in_{\mathbf{W}}\gamma\right)$$
Hence $\mathbf{W}\vDash \mathbf{Pair}$.

\subsubsection{Union}
Thanks to Lemma \ref{rest} we know that
$$\left\|\mathbf{Union}\right\|\equiv_{\Sigma}\lar{\forall} _{\alpha\in \mathbf{W}}\lar{\exists} _{\beta\in \mathbf{W}}\lar{\forall} _{u\in \partial_{0}(\alpha)}\left(\alpha(u)\rightarrow\lar{\forall} _{w\in \partial_{0}(u)}\left(u(w)\rightarrow w\in_{\mathbf{W}}\beta\right)\right)$$
Let us fix now an arbitrary $\alpha\in \mathbf{W}$ and let us define $\zeta_{\alpha}\in \mathbf{W}$ as follows. The domain of $\zeta_{\alpha}$ is $\bigcup_{u\in \partial_{0}(\alpha)}\partial_{0}(u)$ and for every element of such domain $\zeta_{\alpha}(u)=\top$.

Let $w\in \bigcup_{u\in \partial_{0}(\alpha)}\partial_{0}(u)$, then $\vdash\mathbf{p}\top \mathbf{\rho}: \top\times w=_{\mathbf{W}}w$ (we are using Lemma \ref{not}), from which it follows that $\vdash\mathbf{e}(\mathbf{p}\top \mathbf{r})): w\in_{\mathbf{W}}\zeta_{\alpha}$. From this it follows that 
$$\vdash\lambda v'.\lambda v.\mathbf{e}(\mathbf{p}\top \mathbf{r})): \lar{\forall} _{u\in \partial_{0}(\alpha)}\left(\alpha(u)\rightarrow\lar{\forall} _{w\in \partial_{0}(u)}\left(u(w)\rightarrow w\in_{\mathbf{W}}\zeta_{\alpha}\right)\right)$$
and thus that 
$$\vdash\mathbf{e}(\lambda v'.\lambda v.\mathbf{e}(\mathbf{p}\top \mathbf{r}))):  \lar{\exists} _{\beta\in \mathbf{W}}\lar{\forall} _{u\in \partial_{0}(\alpha)}\left(\alpha(u)\rightarrow\lar{\forall} _{w\in \partial_{0}(u)}\left(u(w)\rightarrow w\in_{\mathbf{W}}\beta\right)\right)$$
Since $\mathbf{e}(\lambda v'.\lambda v.\mathbf{e}(\mathbf{p}\top \mathbf{r})))$ does not depend on $\alpha$ and it is an element of $\Sigma$, we have 
$$\vdash\mathbf{e}(\lambda v'.\lambda v.\mathbf{e}(\mathbf{p}\top \mathbf{r}))): \lar{\forall} _{\alpha\in \mathbf{W}}\lar{\exists} _{\beta\in \mathbf{W}}\lar{\forall} _{u\in \partial_{0}(\alpha)}\left(\alpha(u)\rightarrow\lar{\forall} _{w\in \partial_{0}(u)}\left(u(w)\rightarrow w\in_{\mathbf{W}}\beta\right)\right)$$
and $\mathbf{W}\vDash \mathbf{Union}$.
\subsubsection{Powerset}
Using Lemma \ref{sub}
$$\left\|\mathbf{Pow}\right\|\equiv_{\Sigma}\lar{\forall} _{\alpha\in \mathbf{W}}\lar{\exists} _{\beta\in \mathbf{W}}\lar{\forall} _{\gamma\in \mathbf{W}}\left(\gamma\subseteq_{\mathbf{W}} \alpha\rightarrow \gamma\in_{\mathbf{W}} \beta\right)$$
Let us consider an arbitrary $\alpha\in \mathbf{W}$ and define $\pi_{\alpha}\in \mathbf{W}$ as that partial function having domain $A^{\partial_0(\alpha)}$ and for which $\pi_{\alpha}(u)=\top$ for every $u$ in the domain. $\pi_\alpha$ is in $\mathbf{W}$ since $\kappa$ is strongly inaccessible. For every $\gamma\in \mathbf{W}$ we also define $\gamma_{\alpha}\in \mathbf{W}$ as follows. The domain of $\gamma_{\alpha}$ is $\partial_{0}(\alpha)\cup \partial_{0}(\gamma)$ and $\gamma_{\alpha}(u):=u\in_{\mathbf{W}}\alpha \times u\in_{\mathbf{W}}\gamma$ for every $u$ in the domain. 

We now use Lemma \ref{not}.
Let $u\in \partial_{0}(\alpha)$. Clearly:
\begin{enumerate}
\item $x: \gamma\subseteq_{\mathbf{W}}\alpha, y: \gamma(u) \vdash xy: u\in_{\mathbf{W}}\alpha$, 
\item $x: \gamma\subseteq_{\mathbf{W}}\alpha, y: \gamma(u) \vdash \mathbf{j}y: u\in \gamma$, 
\item $x: \gamma\subseteq_{\mathbf{W}}\alpha, y: \gamma(u) \vdash \mathbf{\rho}: u=_{\mathbf{W}}u$. 
\end{enumerate}
 From this it follows that 
$$x: \gamma\subseteq_{\mathbf{W}}\alpha, y: \gamma(u) \vdash \mathbf{r}:=\mathbf{p}(\mathbf{p}(xy)(\mathbf{j}y))\mathbf{\rho}: (u\in_{\mathbf{W}}\alpha \times u\in_{\mathbf{W}}\gamma)\times u=_{\mathbf{W}}u$$
and thus that 
$x: \gamma\subseteq_{\mathbf{W}}\alpha, y: \gamma(u) \vdash \mathbf{er}: u\in \gamma_{\alpha}$ from which it follows that 
$$x: \gamma\subseteq_{\mathbf{W}}\alpha\vdash\lambda y.\mathbf{er}: \gamma(u)\rightarrow u\in \gamma_{\alpha}$$
and since $\lambda y.\mathbf{er}$ does not depend on $u\in \partial_{0}(\alpha)$ 
$$x: \gamma\subseteq_{\mathbf{W}}\alpha\vdash\lambda y.\mathbf{er}: \gamma\subseteq_{\mathbf{W}} \gamma_{\alpha}$$
One can also easily show that $\vdash \lambda z.(\mathbf{p}_2z): \gamma_{\alpha}\subseteq_{\mathbf{W}} \gamma$.
Thus 
$$x: \gamma\subseteq_{\mathbf{W}}\alpha\vdash\overline{\mathbf{r}}:=\mathbf{p}\top(\mathbf{p}(\lambda z.(\mathbf{p}_2z))(\lambda y.\mathbf{er})): \top \times \gamma_{\alpha}=_{\mathbf{W}}\gamma$$
Since $\gamma_{\alpha}$ is in the domain of $\pi_\alpha$ we hence have that 
$$x: \gamma\subseteq_{\mathbf{W}}\alpha\vdash\mathbf{e}\overline{\mathbf{r}}: \gamma \in \pi_{\alpha}$$
We can thus conclude that $\vdash\lambda x.\mathbf{e}\overline{\mathbf{r}}:\gamma\subseteq_{\mathbf{W}}\alpha\rightarrow \gamma\in_{\mathbf{W}}\pi_{\alpha}$.
Since $\lambda x.\mathbf{e}\overline{\mathbf{r}}$ and $\mathbf{e}(\lambda x.\mathbf{e}\overline{\mathbf{r}})$ do not depend on $\gamma$ and $\alpha$ we get, 
$$\vdash\mathbf{e}(\lambda x.\mathbf{e}\overline{\mathbf{r}}): \lar{\forall} _{\alpha\in \mathbf{W}}\lar{\exists} _{\beta\in \mathbf{W}}\lar{\forall} _{\gamma\in \mathbf{W}}\left(\gamma\subseteq_{\mathbf{W}} \alpha\rightarrow \gamma\in_{\mathbf{W}} \beta\right)$$
Since $\mathbf{e}(\lambda x.\mathbf{e}\overline{\mathbf{r}})\in \Sigma$, we can conclude that $\mathbf{W}\vDash \mathbf{Pow}$.

\subsubsection{Infinity}
For every $n\in \omega$, we define $\widehat{n}\in \mathbf{W}$ as follows: $\partial_0(\widehat{n})=\{\widehat{m}|\,m<n\}$ and $\widehat{n}(\widehat{m}):=\overline{m}$ where $\overline{m}\in \Sigma$ is Krivine's encoding of the natural number $m$.
We define $\widehat{\omega}$ as the element of $\mathbf{W}$ with domain $\{\widehat{n}|\,n\in \omega\}$ and defined by $\widehat{\omega}(\widehat{n}):=\overline{n}$.

First, if we consider $\widehat{0}=\emptyset$, we can easily see that $\vdash\mathbf{e}(\mathbf{p}\overline{0}\top):\left\|\mathbf{Inf}_1(u)[u]\right\|(\widehat{\omega})$.

Consider now a closed $\lambda$-term $f\in \Sigma$, which exists by a recursion theorem (and which is in $\Sigma$, because $\rho\in \Sigma$), such that 
$$\begin{cases}f\overline{n}\overline{m}\twoheadrightarrow_{\beta}\mathbf{j}_1(\mathbf{p}\overline{m}\rho)\textrm{ if }\overline{m}\neq \overline{n}\\ f\overline{n}\overline{m}\twoheadrightarrow_{\beta}\mathbf{j}_2\rho \textrm{ if }\overline{m}=\overline{n}\end{cases}$$
for every $n,m\in \omega$.

Then, for every $n\in \omega$
$$\vdash \lambda u.f\overline{n}u:\lar{\forall} _{i=0,...,n}(\widehat{n+1}(\widehat{i})\rightarrow (\widehat{i}\in_{\mathbf{W}}\widehat{n}+\widehat{i}=_{\mathbf{W}}\widehat{n}))$$
Moreover
$$\vdash \lambda x.\mathbf{e}(\mathbf{p}x\rho):\widehat{n}\subseteq_{\mathbf{W}}\widehat{n+1}$$
$$\vdash \mathbf{e}(\mathbf{p}\overline{n}\rho):\widehat{n}\in_{\mathbf{W}}\widehat{n+1}$$
Thus $t\leq \left\|\mathbf{Inf}_{2}(u)[u]\right\|(\widehat{\omega})$ for some $t\in \Sigma$, and we can conclude that $\mathbf{W}\vDash \mathbf{Inf}$.

\subsubsection{Separation}
Assume $\varphi\,[\underline{w},x,z]$ be a formula in context with $\underline{w}$ a list of variable of length $n$.
{\small $$\left\|\mathbf{Sep}_{\varphi}\right\|\equiv_{\Sigma}\lar{\forall} _{\underline{\omega}\in \mathbf{W}^{n}}\lar{\forall} _{\alpha\in \mathbf{W}}\lar{\exists} _{\beta\in \mathbf{W}}(
\lar{\forall} _{u\in \partial_{0}(\beta)}(\beta(u)\rightarrow u\in_{\mathbf{W}}\alpha\times \left\|\varphi\,[\underline{w},x,z]\right\|(\underline{\omega},\alpha,u))\times$$ }
$$\lar{\forall} _{u'\in \partial_{0}(\alpha)}(\alpha(u')\rightarrow ( \left\|\varphi\,[\underline{w},x,z]\right\|(\underline{\omega},\alpha,u')\rightarrow u'\in_{\mathbf{W}}\beta)))$$
For an arbitrary $\alpha\in \mathbf{W}$ and $\underline{\omega}\in \mathbf{W}^n$ we define $\alpha_{\varphi}^{\underline{\omega}}\in \mathbf{W}$ as follows: its domain is equal to the domain of $\alpha$, while $\alpha_{\varphi}^{\underline{\omega}}(u):=\alpha(u)\times \left\|\varphi\,[\underline{w},x,z]\right\|(\underline{\omega},\alpha,u)$.
In order to show that $\mathbf{W}\Vdash \mathbf{Sep}_{\varphi}$, it is sufficient to find a $t\in \Sigma$ not depending on $\overline{\omega}$ and $\alpha$ such that 
$$\vdash t: \lar{\forall} _{u\in \partial_{0}(\alpha)}(\alpha_{\varphi}^{\overline{\omega}}(u)\rightarrow u\in_{\mathbf{W}}\alpha\times \left\|\varphi\,[\underline{w},x,z]\right\|(\underline{\omega},\alpha,u))\times$$ 
$$\lar{\forall} _{u'\in \partial_{0}(\alpha)}(\alpha(u')\rightarrow ( \left\|\varphi\,[\underline{w},x,z]\right\|(\underline{\omega},\alpha,u')\rightarrow u'\in_{\mathbf{W}}\alpha_{\varphi}^{\overline{\omega}}))$$
But this is immediate to prove, since using Lemma \ref{not}
$$\vdash\lambda x.\mathbf{p}(\mathbf{j}(\mathbf{p}_1x))(\mathbf{p}_2x): \lar{\forall} _{u\in \partial_{0}(\alpha)}(\alpha_{\varphi}^{\overline{\omega}}(u)\rightarrow u\in_{\mathbf{W}}\alpha\times \left\|\varphi\,[\underline{w},x,z]\right\|(\underline{\omega},\alpha,u))$$
$$\vdash\lambda x.\lambda y.\mathbf{e}(\mathbf{p}(\mathbf{p}xy)\mathbf{\rho}): \lar{\forall} _{u'\in \partial_{0}(\alpha)}(\alpha(u')\rightarrow ( \left\|\varphi\,[\underline{w},x,z]\right\|(\underline{\omega},\alpha,u')\rightarrow u'\in_{\mathbf{W}}\alpha_{\varphi}^{\overline{\omega}}))$$

\subsubsection{$\in$-Induction}
Let $\mathbf{y}$ be the fix-point operator such that $\mathbf{y}f$ $\beta$-reduces to $f(\mathbf{y}f)$ for every $f$ and consider 
$$\mathbf{h}:=\mathbf{y}(\lambda h.\lambda x.x(\lambda y.hx))\in \Sigma$$
in such a way that $\mathbf{h}\leq (\lambda h.\lambda x.x(\lambda y.hx))\mathbf{h}\leq \lambda x.x(\lambda y.\mathbf{h}x)$.

We want to prove that $\mathbf{h}\leq \left\| \in\textrm{-}\mathbf{Ind}_{\varphi}\right\|$ to conclude that $\mathbf{W}\vDash \in\textrm{-}\mathbf{Ind}_{\varphi}$.

We prove this by induction on the rank in $\mathbf{W}$. Fix an arbitrary $\overline{\alpha}$ and assume that 
$$\mathbf{h}\leq \lar{\forall} _{\alpha\in \mathbf{W}}\left(\lar{\forall} _{u\in \partial_{0}(\alpha)}(\alpha(u)\rightarrow \left\|\varphi[x]\right\|(u))\rightarrow \left\|\varphi[x]\right\|(\alpha)\right)\rightarrow \left\|\varphi[x]\right\|(\beta)$$ for every $\beta$ with rank strictly less than that of $\overline{\alpha}$.
Let us use $\varepsilon^{\alpha}$ as a shorthand for $\lar{\forall} _{u\in \partial_{0}(\alpha)}(\alpha(u)\rightarrow \left\|\varphi[x]\right\|(u))\rightarrow \left\|\varphi[x]\right\|(\alpha)$ and $\varepsilon$ as a shorthand for $\lar{\forall} _{\alpha\in \mathbf{W}}\varepsilon^{\alpha}$.

If we consider the following derivation tree:

$$\cfrac{\cfrac{\cfrac{x:\varepsilon\vdash x:\varepsilon}{x:\varepsilon\vdash x:\varepsilon^{\overline{\alpha}}}\qquad \cfrac{\cfrac{\cfrac{\cfrac{\cfrac{\vdash \mathbf{h}:\varepsilon\rightarrow  \left\|\varphi\,[x]\right\|(u)\,(\textrm{ for every }u\in \partial_{0}(\overline{\alpha}))}{x: \varepsilon\vdash \mathbf{h}:\varepsilon\rightarrow  \left\|\varphi\,[x]\right\|(u)(\textrm{ for every }u\in \partial_{0}(\overline{\alpha}))}\qquad x:\varepsilon\vdash x:\varepsilon}{x:\varepsilon\vdash \mathbf{h}x: \left\|\varphi\,[x]\right\|(u)(\textrm{ for every }u\in \partial_{0}(\overline{\alpha}))}}{x:\varepsilon,y: \overline{\alpha}(u)\vdash \mathbf{h}x: \left\|\varphi\,[x]\right\|(u)(\textrm{ for every }u\in \partial_{0}(\overline{\alpha}))}}{x:\varepsilon\vdash \lambda y.\mathbf{h}x: \overline{\alpha}(u)\rightarrow \left\|\varphi\,[x]\right\|(u)(\textrm{ for every }u\in \partial_{0}(\overline{\alpha}))}}{x:\varepsilon\vdash \lambda y.\mathbf{h}x: \lar{\forall} _{u\in \partial_{0}(\overline{\alpha})}(\overline{\alpha}(u)\rightarrow \left\|\varphi\,[x]\right\|(u))}}{x:\varepsilon\vdash x(\lambda y.\mathbf{h}x):\left\|\varphi[x]\right\|(\overline{\alpha})}}{\vdash \lambda x.x(\lambda y.\mathbf{h}x):\varepsilon\rightarrow \left\|\varphi[x]\right\|(\overline{\alpha})}$$

we can conclude that $\mathbf{h}\leq \varepsilon\rightarrow \left\|\varphi[x]\right\|(\overline{\alpha})$.

\subsubsection{Collection}
In order to lighten the notation we will consider $\mathbf{Col}_{\varphi}$ for a formula $\varphi$ in context $[x,y]$ (so without any additional parameter). Moreover we will write $\varphi(a,b)$ instad of $\left\|\varphi\,[x,y]\right\|(a,b)$.

Assume $\alpha\in \mathbf{W}$ and $u\in\partial_0(\alpha)$. Since $\kappa$ in inaccessible, $|A|<\kappa$ and $\{\varphi(u,\gamma)|\,u\in \mathbf{W}\}\subseteq A$, there exists $\eta<\kappa$ such that $\lar{\exists} _{\gamma\in \mathbf{W}}(\top\times \varphi(u,\gamma))=\lar{\exists} _{\gamma\in W^{\mathbb{A}}_{\eta}}(\top \times \varphi(u,\gamma))$. We define $\eta_{u}$ to be the minimum such an $\eta$ and we define $\overline{\eta}_{\alpha}:=\bigvee\{\eta_u|\,u\in \partial_{0}(\alpha)\}$ which is strictly less than $\kappa$, since the cardinality of $A$ is strictly less than $\kappa$. 
We define $\beta_{\alpha}\in \mathbf{W}$ as the constant function with value $\top$ and domain $W^{\mathbf{A}}_{\overline{\eta}_{\alpha}}$.  Using the calculus we can show that there is an element $r\in \Sigma$ not depending on $\alpha$ such that 
$$\vdash r:\lar{\forall} _{u\in \partial_{0}(\alpha)}\left(\alpha(u)\rightarrow \lar{\exists} _{\gamma\in \mathbf{W}}\varphi(u,\gamma)\right)\rightarrow$$
$$\qquad\qquad\qquad\qquad\qquad\qquad\qquad \lar{\forall} _{u\in \partial_{0}(\alpha)}\left(\alpha(u)\rightarrow \lar{\exists} _{w\in\partial_{0}(\beta_{\alpha})}(\beta_{\alpha}(w)\times\varphi(u,w))\right)$$
and using this fact one can easily show that $\mathbf{Col}_{\varphi}$ is validated in the model.

\section{Models of $\mathbf{IZF}$ in a class of toposes}
\begin{theorem} Every topos $\mathcal{E}$ obtained from an implicative tripos by means of the tripos-to-topos construction from an implicative algebra $\mathbb{A}=(A,\leq,\rightarrow,\Sigma)$ with $|A|<\kappa$ for a strongly inaccassible cardinal $\kappa$ hosts a model of $\mathbf{IZF}$. If $\Sigma$ is classical (see \cite{miq1}), then $\mathcal{E}$ hosts a model of $\mathbf{ZF}$.
\end{theorem}
\begin{proof}
In such a topos $\mathcal{E}$ an internal model of $\mathbf{IZF}$ is given by the object $(\mathbf{W},[=_{\mathbf{W}}])$ together with the mono embedding the object $$(\mathbf{W}\times \mathbf{W},[((x,y),(x',y'))\mapsto x\in_{\mathbf{W}}y \times x=_{\mathbf{W}}x' \times y=_{\mathbf{W}}y'])$$ into $(\mathbf{W},[=_{\mathbf{W}}])\times (\mathbf{W},[=_{\mathbf{W}}])$ which inteprets the membership relation.
\end{proof}
Using Theorem \ref{teomiq}, we obtain the following:
\begin{corollary} If for every cardinal $\kappa'$ there exists a strongly inaccessible cardinal $\kappa$ such that $\kappa'<\kappa$ (i.e. if the inaccessible cardinal axiom holds), then every topos obtained from a set-based tripos by means of the tripos-to-topos construction hosts a model of $\mathbf{IZF}$.
\end{corollary}

\subsubsection*{Acknowledgements} The author would like to acknowledge T. Streicher and F.Ciraulo for useful discussions on the subject of this paper.

\bibliographystyle{alpha}
\bibliography{biblioPAS}

\end{document}